\documentclass[a4paper,10pt]{article}

\usepackage[pdftex]{graphicx}
\usepackage[section] {placeins}
\usepackage{enumerate}
\usepackage[normalem]{ulem}
\usepackage{amsmath}
\usepackage{amsthm}
\usepackage{latexsym}
\usepackage{amsfonts}
\usepackage{amssymb}
\usepackage{mathtools}
\usepackage{longtable}

\usepackage{tikz-cd}
\usetikzlibrary{graphs}
\usetikzlibrary{arrows.meta}
\usetikzlibrary{shapes.geometric}

\usepackage{authblk}

\usepackage{latexsym,amssymb,amsfonts, amsthm, amsmath}
\usepackage[english]{babel}
\usepackage{bm}
\usepackage{mathrsfs}
\usepackage{algpseudocode}
\usepackage{algorithm}

\algnewcommand{\IIf}[1]{\State\algorithmicif\ #1\ \algorithmicthen}
\algnewcommand{\EndIIf}{\unskip\ \algorithmicend\ \algorithmicif}

\newtheorem{thm}{Theorem}[section]
\newtheorem{lem}[thm]{Lemma}
\newtheorem{cor}[thm]{Corollary}
\newtheorem{prop}[thm]{Proposition}

\theoremstyle{definition}
\newtheorem{defn}[thm]{Definition}

\newtheorem{rem}[thm]{Remark}

\DeclareMathOperator\supp{supp}

\usepackage[margin=3.54cm]{geometry}

\makeatletter
\newcommand{\subjclass}[2][1991]{%
\let\@oldtitle\@title%
\gdef\@title{\@oldtitle\footnotetext{#1 \emph{Mathematics subject classification.} #2}}%
}
\newcommand{\keywords}[1]{%
\let\@@oldtitle\@title%
\gdef\@title{\@@oldtitle\footnotetext{\emph{Key words and phrases.} #1.}}%
}
\makeatother

\title{Bounds on the Higher Degree \\ Erd\H{o}s-Ginzburg-Ziv Constants over $\mathbb{F}_q^n$}

\date{}
\author[1]{Simone Costa}
\author[2]{Stefano Della Fiore}

\affil[1]{DICATAM, Sez.~Matematica, Universit\`a degli Studi di Brescia,  Via Branze~43, I~25123 Brescia, Italy}
\affil[2]{DI, Universit\`a degli Studi di Salerno, Via Giovanni Paolo II 132, 84084 Fisciano, Italy}
\begin{document}
\maketitle
\begin{abstract}
\noindent
The classical Erd\H{o}s-Ginzburg-Ziv constant of a group $G$ denotes the smallest positive integer $\ell$ such that any sequence $S$ of length at least $\ell$ contains a zero-sum subsequence of length $\exp(G)$.

In the recent paper, \cite{CS}, Caro and Schmitt generalized this concept, using the $m$-th degree symmetric polynomial $e_m(S)$ instead of the sum of the elements of $S$ and considering subsequences of a given length $t$. In particular, they defined the higher degree Erd\H{o}s-Ginzburg-Ziv constants $EGZ(t,R,m)$ of a finite commutative ring $R$ and presented several lower and upper bounds to these constants.

This paper aims to provide lower and upper bounds for $EGZ(t,R,m)$ in case $R=\mathbb{F}_q^{n}$. The lower bounds here presented have been obtained, respectively, using Lov\'asz Local Lemma and the Expurgation method and, for sufficiently large $n$, they beat the lower bound provided by Caro and Schmitt for the same kind of rings.  Finally,  we prove closed form upper bounds derived from the Ellenberg–Gijswijt and Sauermann results for the cap-set problem assuming that $q = p^k$,  $t = p$,  and $m=p-1$.  Moreover,  using the Slice Rank method we derive a convex optimization problem that provides the best bounds for $q = 3^k$,   $t = 3$,  $m=2$ and $k=2, 3,4,5$.
\end{abstract}
\section{Introduction}\label{sec:intro}
One significant subfield of additive group theory and combinatorial number theory is the zero-sum theory that studies the sums behavior of suitable sequences of elements in an abelian finite group $G$ (see, for instance, the surveys \cite{C, GG}). In this context, a typical kind of problem considers the existence of constants $\ell$ such that any sequence of elements of $G$ whose length is bigger than $\ell$ satisfies an additive property $\mathcal{P}$. Among these constants, an important role is taken by the classical Erd\H{o}s-Ginzburg-Ziv constant of a group $G$ that denotes the smallest positive integer $\ell$ such that any sequence of length $|S|\geq \ell$ contains a zero-sum subsequence of length $\exp(G)$. This constant has been well studied in the literature, we refer to the survey paper \cite{GG}. Here we recall that in \cite{EGZ}, Erd\H{o}s, Ginzburg and Ziv completely determined its value over cyclic groups and that in \cite{FS},  \cite{N} and \cite{Sauer}, respectively Fox, Sauermann and Naslund, derived nontrivial upper bounds on groups of type $\mathbb{F}_p^n$ (they assumed a slightly different definition of Erd\H{o}s-Ginzburg-Ziv constant).

In the recent paper \cite{CS}, Caro and Schmitt generalized this concept, using the $m$-th degree symmetric polynomial $e_m(g_1,\dots,g_{t})=\sum_{1\leq i_1<\cdots<i_m\leq t} \prod_{j=1}^m g_{i_j}$ instead of the sum of the elements of $S$ and considering subsequences of a given length $t$ (see also \cite{ABV, BL1, BL2} and \cite{Bi} that considered some related problems). In particular, they defined the higher degree Erd\H{o}s-Ginzburg-Ziv constants $EGZ(t,R,m)$ as follows.
For a finite commutative ring $R$, $EGZ(t,R,m)$ is the smallest positive integer $\ell$ such that every sequence $S$ over $R$ of length $|S|\geq \ell$ contains a subsequence $S'$ of length $t$ for which $e_{m}(S')$ evaluates to the zero-element in $R$. If such $\ell$ does not exists, $EGZ(t,R,m)$ is set to $\infty$.

They also present several lower and upper bounds to these constants solving the case where $R$ is $\mathbb{Z}_2$ and the case where $R$ is $\mathbb{Z}_{p^{s}}$ if $t$ and $m$ are powers of the same prime. For a generic finite commutative ring $R$,  their best lower bound is expressed in term of the generalized Davenport constant $D(R,m)$ of the ring $R$ (see Caro, Girard and Schmitt, \cite{CGS}) that is the smallest integer $\ell$ such that any sequence $S$ over $R$ of length $|S|\geq \ell$ contains a subsequence $S'$ of length $|S'|\geq m$ for which $e_m(S')$ equals the zero element of $R$. Indeed they prove that \begin{equation}\label{LowerCS}EGZ(t,R,m)\geq t+D(R,m)-m.\end{equation}

This paper aims to determine lower and upper bounds for $EGZ(t,R,m)$ in case $R=\mathbb{F}_q^{n}$ (viewed as a commutative ring) for some prime power $q$ (in the following we will always use the letter $q$ for a prime power and $p$ for a prime). The article is organized as follows. In Section 2 we will present two lower bounds obtained, respectively, using Lov\'asz Local Lemma and the Expurgation method. 
Then, in Section 3, we will show that, for sufficiently large $n$, our bounds improve the ones given by Caro and Schmitt in the same context.
Finally, in Section 4,  we prove closed form upper bounds to $EGZ(p,R,p-1)$,  derived from the Ellenberg–Gijswijt \cite{ellenberg-gjswijt-2017} and Sauermann \cite{Sauer} bounds for the cap-set problem,  in case $R=\mathbb{F}_q^{n}$ and $q=p^k$.  Moreover,  we will apply Tao's Slice Rank method to provide an upper bound to $EGZ(3,\mathbb{F}_q^n,2)$ and we derive a convex optimization problem that we can solve numerically providing better bounds for $q = 3^k$ and $k=2, 3,4,5$.
\section{Lower bounds}\label{sec:LB}
In this section, we will present two kinds of probabilistic lower bounds on the Erd\H{o}s-Ginzburg-Ziv constants of rings of type $\mathbb{F}_q^n$. Both those bounds exploit the following upper bound on the probability that a given $t$-sequence $S$ of elements in (vectors of) $\mathbb{F}_q^n$ is such that $e_m(S)=0$.  To provide such an upper bound, we exploit the following famous lemma.

\begin{lem}[Schwartz-Zippel Lemma]\label{lem:SZ}
Let $P \in \mathbb{F}[x_1, x_2, \ldots, x_t]$ be a non-zero polynomial with degree $d$.  Consider a finite subset $A \subseteq \mathbb{F}$.  If we pick uniformly at random $r_1, r_2, \ldots, r_t$ from $A$,  then
$$
	\mathbb{P}[P(r_1,  r_2,  \ldots,  r_t) = 0] \leq \frac{d}{|A|}\,.
$$
\end{lem}
Using Lemma \ref{lem:SZ},  we easily obtain the follow proposition.
\begin{prop}\label{bound}
Let us choose, uniformly at random, a sequence $S = (g_1, g_2, $ $\ldots, g_t)$ of $t\geq m$ vectors of $\mathbb{F}_q^n$. Then
$$\mathbb{P}[e_m(S)=0]\leq \left(\frac{m}{q}\right)^{n}.$$
\end{prop}
\proof
We prove this result first assuming $n=1$.  Since $e_m(S)$ is a polynomial of degree $m$,  by Lemma \ref{lem:SZ},  taking $A = \mathbb{F}_q$,  we obtain

$$
	\mathbb{P}[e_m(S)=0] \leq \frac{m}{|A|} = \frac{m}{q}\,.
$$

%
%
%
%
Now we note that, if we consider a sequence $S$ of $t\geq m$ vectors of $\mathbb{F}_q^n$, then ${e_m(S)=0}$ if and only if each of the $n$ projections $\pi_i(S)$ of $S$ over the $i$-th coordinate satisfies ${e_m(\pi_i(S))=0}$. Since those projections are independent, it follows that
$$\mathbb{P}[e_m(S)=0]=\prod_{i=1}^{n}\mathbb{P}[e_m(\pi_i(S))=0]\leq \left(\frac{m}{q}\right)^{n}.$$
\endproof

We provide a first new lower bound on $EGZ(t,\mathbb{F}_q^{n},m)$ by exploiting the so-called Lov\'asz Local Lemma, see also the work \cite{Bi} of Bitz, Griffith and He for a similar application of this method. Here we state the lemma (in the symmetric case) for the reader's convenience.

\begin{lem}[\cite{LLL} (see also \cite{AS})]\label{lem:LLL}
Let $E_1, E_2, \ldots, E_k$ be events in an arbitrary probability space. Suppose that each event $E_i$ is mutually independent of the set of all other events $E_j$ but at most $d$, and that $\mathbb{P}[E_i] \leq P$ for all $1 \leq i \leq k$. If
$$
e d P \leq 1
$$
then $\mathbb{P}[\cap_{i=1}^k \overline{E_i}] > 0$.
\end{lem}

Now, we are ready to state the following theorem.

\begin{thm}\label{thm:LLL}
Let $\ell$ be such that
$$ e \left[ {\ell \choose t} - { \ell-t \choose t} \right] \left(\frac{m}{q}\right)^{n} \leq 1$$
where ${\ell-t \choose t}$ is set to zero if $\ell < 2t$.
Then $EGZ(t,\mathbb{F}_q^{n},m) > \ell.$
\end{thm}

\proof
Here we need to prove the existence of a sequence $S$ of length $\ell$ for which any subsequence $S'$ of length $t$ is such that $e_{m}(S')\not =0$.

Let us choose, uniformly at random, a sequence $S$ of length $\ell$ in $\mathbb{F}_q^{n}$. For a given subsequence $S'$ of length $t$ contained in $S$, let $E_{S'}$ be the event such that $e_{m}(S') =0$. Clearly, there are ${\ell \choose t}$ such events. Due to Proposition \ref{bound}, we know that
$$
\mathbb{P}[E_{S'}] \leq \left(\frac{m}{q}\right)^{n} \quad \text{ for all } \quad S' \subseteq S, |S'| = t.
$$
It is easy to see that each event $E_{S'}$ is mutually independent from all the events $E_{S''}$ where $S'' \subseteq S \setminus S'$ and $|S''| = t$. Therefore each event $E_{S'}$ is dependent by at most ${\ell \choose t} - {\ell-t \choose t}$ other events. Hence, due to Lemma \ref{lem:LLL} we obtain the thesis.
\endproof

Now we provide a second lower bound that, in some regime of the parameters turns out to improve that of Section 2.1. The method we use here is sometimes called Expurgation in the literature. We refer the reader to the book \cite[Chapter 3 (Alterations)]{AS}.
\begin{thm}\label{thm:exp}
Let $\ell$ be such that
$$ {\ell+s\choose t}\left(\frac{m}{q}\right)^{n}<s+1 $$
for some $s\geq 0$. Then $EGZ(t,\mathbb{F}_q^{n},m) > \ell.$
\end{thm}
\proof
We first note that the thesis is equivalent to prove the existence of a sequence $S$ of length $\ell$ for which any subsequence $S'$ of length $t$ is such that ${e_{m}(S')\not =0}$.

Here we choose, uniformly at random, a sequence $T$ of length $\ell+s$ and we evaluate the expected value of the random variable $X$ given by the number of subsequences $T'$ of $T$ of length $t$ and such that $e_{m}(T')=0.$
Because of Proposition \ref{bound}, we have that
$$\mathbb{E}(X)\leq \sum_{T'\subseteq T: |T'|=t}\left(\frac{m}{q}\right)^{n}={\ell+s\choose t}\left(\frac{m}{q}\right)^{n}.$$
Moreover, due to the hypothesis, we have that $\mathbb{E}(X)<s+1$.
It follows that there exists a set $T$ of length $\ell+s$ with at most $s$ subsequences $T'$ such that $e_m(T')=0$ that we call bad subsequences. If we remove from $T$ one element from each bad subsequence we have removed at most $s$ elements and we are left with a sequence $S$ of length at least $\ell$ for which any subsequence $S'$ of length $t$ is such that $e_{m}(S')\not =0.$ Clearly, we may assume, without loss of generality, that the length of $S$ is exactly $\ell$ obtaining the thesis.
\endproof

\begin{rem}
We have been able to compute the optimal value of $s$ in the expurgation bound (Theorem \ref{thm:exp}) only for small values of $t$, i.e., $t=2,3,4,5$. In all these cases the expurgation bound performs better than the bound given in Theorem \ref{thm:LLL} obtained using the Lov\'asz Local Lemma. In view of these results, we are inclined to conjecture that the expurgation bound provides the best bound for every $\ell \geq t \geq 2$. However, since we did not succeed to prove this conjecture, we considered useful to report also the bound given in Theorem \ref{thm:LLL}.
\end{rem}

\section{Comparation with Caro and Schmitt's bounds}
In this section, we discuss the bounds we have obtained in comparison to that of Caro and Schmitt. In particular, in \cite{CGS} (see Theorems 3.1 and 3.4), it was proved that for rings of type $\mathbb{F}_p^{n}$ (where $p$ is a prime) the following bounds on $D(\mathbb{F}_p^{n},m)$ hold:
\begin{equation}\label{BoundsD} nmp-(n-1)m \geq D(\mathbb{F}_p^{n},m)\geq n p-(n-1)m.\end{equation}
It follows that the lower bound of Caro and Schmitt \eqref{LowerCS} becomes
\begin{equation}\label{LowerCS2}EGZ(t,\mathbb{F}_p^{n},m)\geq t+n (p-m).\end{equation}
Now we consider our lower bound of Theorem \ref{thm:exp} with $s=0$ in the case $q=p$. Note that this is not, in general, our best lower bound but it is the easiest to consider. We have that $EGZ(t,\mathbb{F}_p^{n},m)\geq \ell$, if $\ell$ is such that
$${\ell\choose t}\left(\frac{m}{p}\right)^{n}<1.$$
We note that $\frac{\ell^t}{t!}>{\ell\choose t}$ and hence $EGZ(t,\mathbb{F}_p^{n},m)\geq \ell$ for any $\ell$ such that
$$\frac{\ell^t}{t!}\left(\frac{m}{p}\right)^{n}<1$$
that is
$$\frac{\ell^t}{t!}<\left(\frac{p}{m}\right)^{n}$$
and hence
\begin{equation}\label{eq1}EGZ(t,\mathbb{F}_p^{n},m)\geq (t!)^{\frac{1}{t}}\left(\frac{p}{m}\right)^{\frac{n}{t}}.\end{equation}
Now, since \eqref{eq1} is,  when $p>m$, exponential in $n$, it is clear that asymptotically in $n$, it improves the lower bound of equation \eqref{LowerCS2}.
\begin{rem}
From equation \eqref{eq1}, we also have that,  if $p>m$, for sufficiently large $n$:
$$EGZ(t,\mathbb{F}_p^{n},m)\geq (t!)^{\frac{1}{t}}\left(\frac{p}{m}\right)^{\frac{n}{t}}> t+nm(p-1)\geq t+D(\mathbb{F}_p^n,m)-m$$
where the last inequality follows from the upper bound of equation \eqref{BoundsD}.
This means that, for these kinds of parameters it does not yield a Caro-Gao-type relation (see \cite{C2,C3,G}), i.e. it does not hold the equality in equation \eqref{LowerCS}. 
\end{rem}
We also note that the bound of equation \eqref{eq1} can be trivially improved for several values of $q=p^k$. Indeed, if ${t\choose m}\not \equiv 0 \pmod{p}$ we have that $EGZ(t,\mathbb{F}_q^{n},m)=\infty$. It suffices to consider the infinite constant sequence such that $g_i=1$ for any $i\in \mathbb{N}$.
In this case we have that, for any subsequence $S'$ of length $t$, $e_m(S')={t\choose m}\not\equiv 0\pmod{p}$. On the other hand, this is a subset of the parameters for which our bounds of Section 2 (and in particular equation \eqref{eq1}) hold.  Moreover,  we will show in the upcoming section that, at least when $q=p^k$,  $t=p$ and $m=p-1$,  it is possible to provide nontrivial upper bounds.

Finally, we also note that the bounds here presented can be easily generalized to rings of type $\mathbb{F}_{q_1}\times \mathbb{F}_{q_2}\times \cdots \times \mathbb{F}_{q_n}$ (similarly to those of \cite{CS} that Caro and Schmitt stated for rings of type $\mathbb{Z}_{m_1}\times \mathbb{Z}_{m_2}\times \cdots \times \mathbb{Z}_{m_n}$) but, since we believe this is not a substantial improvement, we prefer to keep the notation of this note as simple as possible and to explicitly consider only rings of type $\mathbb{F}_q^n$.
\section{Upper bounds}
In this section we provide an upper bound to $EGZ(p,\mathbb{F}_{q}^{n},p-1)$ with $q=p^k$,  where $p$ is an odd prime.  In the first part of this section, we provide using the Ellenberg–Gijswijt \cite{ellenberg-gjswijt-2017} and Sauermann \cite{Sauer} bounds for the cap-set problem a general upper bound to $EGZ(p,\mathbb{F}_{q}^{n},p-1)$ for every prime $p \geq 3$. Then,  we use the so-called Slice Rank method, introduced by Terence Tao in \cite{Blog1} and revisited by Tao and Sawin in \cite{Blog2} (see also \cite{CostaDalai2021} for a discussion on the method) in order to generalize the polynomial approach introduced in \cite{CLP} and in \cite{ellenberg-gjswijt-2017},  to improve the bounds that can be deduced by the Ellenberg–Gijswijt bound for $p=3$ and $k = 2, 3, 4, 5$.
Our application of the method is somehow reminiscent of works on the classical Erd\H{o}s-Ginzburg-Ziv constants of Fox and Sauermann \cite{FS} and Naslund \cite{N}.

Let us first state the following theorems that will be used in Theorem \ref{thm:generalBound} to provide a general bound on $EGZ(p,\mathbb{F}_{q}^{n},p-1)$ for $q=p^k$ and $k \geq 2$.

\begin{thm}[Ellenberg–Gijswijt \cite{ellenberg-gjswijt-2017}]\label{thm:EG}
Let $A$ be a subset of $\mathbb{F}_3^n$ which does not contain $3$ distinct elements $x_1$,  $x_2$,  $x_3$ such that $x_1+x_2+x_3 = 0$.  Then,  for $n \to \infty$,  we have that
$$
	|A| \leq \left(\frac{3}{8} \sqrt[3]{207 + 33 \sqrt{33}} + o(1) \right)^n \approx (2.756 + o(1))^n\,.
$$
\end{thm}

\begin{thm}[Sauermann \cite{Sauer}]\label{thm:Sauer}
Let $p \geq 5$ be a prime and let $A$ be a subset of $\mathbb{F}_p^n$ which does not contain $p$ distinct elements $x_1$,  $x_2$,  $\ldots$,  $x_p$ such that $x_1+x_2+\ldots+x_p = 0$.  Then,  for $n \to \infty$,  we have that
$$
	|A| \leq \left(2 \sqrt{p} + o(1)\right)^n.
$$
\end{thm}

Now we consider a sequence $S=(g_1,g_2,\dots,g_\ell)$ of elements in $\mathbb{F}_{q}^{n}$ with $q=p^k$ such that every $p$-tuple of elements $g_1',g_2',\ldots, g_p'$ of $S$ satisfies $e_{p-1}(g_1',g_2',\ldots, g'_p)$ $=\sum_{1\leq i_1<\cdots<i_{p-1}\leq p}$ $\prod_{j=1}^{p-1} g'_{i_j} \not=0$. We note that $S$ can not have elements repeated more than $p-1$ times since $e_{p-1}(g_1',g_2',\ldots, g'_p) = 0$ whenever $g_1'=g_2'=\ldots = g_p'$. It means that we can remove the repeated elements in $S$ obtaining a set $S_1$ with $|S_1|\geq \frac{|S|}{p-1}$. 
Clearly, to upper bound the length of the sequence $S$, it suffices to bound the cardinality of $S_1$ considered as a set (it has no repetitions). Since it does not admit repeated elements, we already have that
\begin{equation}\label{UpperFor3}
\frac{|S|}{p-1} \leq |S_1|\leq q^n.
\end{equation}

Now, we are ready to state our first result whose proof has been pointed out by an anonymous referee.

\begin{thm}\label{thm:generalBound}
Let $p$ be an odd prime and $k$ a positive integer.  Then we have that when $q$ is not a  power of $p$
$$
	EGZ(p,  \mathbb{F}_q^n, p-1) = \infty\,.
$$
While for $q = p^k$,  we have that
$$
EGZ(p,  \mathbb{F}_q^n, p-1) \leq \begin{cases}
q^{n + o(n)} &\text{ for } p=3,5 \text{ and } k=1 \\
(2.756^k +1)^{n + o(n)} &\text{ for } p=3 \text{ and } k \geq 2 \\
\left(2^k \sqrt{q}+1 \right)^{n+o(n)} &\text{ for } p \neq 3 \text{ and } q \geq 7 \\
\end{cases}\,.
$$
\end{thm}
\begin{proof}
We note that, if $\mathbb{F}_q$ has characteristic $p' \neq p$,  ${p \choose p-1}=p \not\equiv 0 \pmod{p'}$. In this case we consider the infinite constant sequence such that $g_i=1$ for any $i\in \mathbb{N}$. Here we have that, for any subsequence $S'$ of length $p$, $e_{p-1}(S')={p\choose p-1}=p\not\equiv 0\pmod{p'}$ and hence $EGZ(p,\mathbb{F}_{q}^{n},p-1)=\infty$ whenever $q$ is not a power of $p$.  Therefore we can assume that $q = p^k$.  For $k=1$ and $p = 3, 5$ then the upper bound of this theorem follows directly from equation \eqref{UpperFor3}.  Hence we can suppose that $k \geq 2$. 

Now,  let $S \subseteq \mathbb{F}_q^n$ be a subset not containing $p$ distinct elements $x_1$, $x_2$, $\ldots$, $x_{p} \in S$ such that $e_{p-1} (x_1, x_2, \ldots, x_p) = 0$.  For every $P \subseteq \{1, 2, \ldots, n\}$,  let us denote with $S_P = \{ v \in S\:|\:\supp(v) = P \}$,  where the $\supp(v)$ is defined as the set of coordinates in which $v$ is nonzero,  the set of vectors in $S$ that have the same support $P$.  Let also denote with $S'_P\subseteq \mathbb{F}_q^{|P|}$ the set obtained from $S_P$ restricting every vector $v \in S_P$ only to coordinates in $P$.  This guarantees us that all the entries of $S_P'$ are nonzero elements of $\mathbb{F}_q$.  Then we construct a new set $S_P'' \subseteq \mathbb{F}_q^{|P|}$ by replacing every vector $(a_1,  a_2, \ldots, a_{|P|}) \in S_P'$ by $(a_1^{-1}, a_2^{-1}, \ldots,  a_{|P|}^{-1}) \in S_P''$.  Clearly,  $|S_P| = |S_P'| = |S_P''|$.  We claim that $S_P''$ does not contain $p$ distinct elements summing to zero in $\mathbb{F}_q^{|P|}$.  Indeed,  suppose by contradiction there exist $p$ distinct vectors $x_1,  x_2,  \ldots,  x_p \in S_P'$ such that $x_{1, i}^{-1} + x_{2,i}^{-1} + \ldots + x_{p, i}^{-1} = 0$ for every $i$ then we can multiply both sides of the previous equation by $x_{1, i} x_{2, i} \cdots x_{p, i}$ to obtain that $e_{p-1} (x_1,  x_2, \ldots, x_p) = 0$.  But this is absurd due to the initial hypothesis on $S$.

Since, as an abelian group under addition,  $\mathbb{F}_q^{|P|}$ is isomorphic to $\mathbb{F}_p^{k |P|}$,  by Theorem \ref{thm:EG} and \ref{thm:Sauer} we have that,  for fixed $p$ and $k$,  $|S_P| \leq \left(u_{p, k} +o(1) \right)^{|P|}$ for every $P \subseteq \{1, 2,\ldots, n\}$ such that $|P| = \alpha n (1+o(1))$ for some $0 < \alpha < 1$, where
$$
u_{p, k} := \begin{cases} 2.756^{k} &\text{for } p=3 \text{ and } k \geq 2 \\
	(2 \sqrt{p})^{k} &\text{for } p \neq 3 \text{ and } p^k \geq 7	
	 \end{cases}.
$$
Hence,  for any real $0 < \alpha \leq 1/4$,  we get
\begin{align*}
|S| &\leq \sum_{\substack{P \subseteq \{1, 2,\ldots, n\} \\ |P|  \leq \alpha n}} |S_P| + \sum_{\substack{P \subseteq \{1, 2,\ldots, n\} \\ |P| \geq \alpha n}} |S_P| \stackrel{(i)}{\leq} 2^n q^{\alpha n} + \sum_{i \geq \alpha n}^{n} \sum_{\substack{P \subseteq \{1, 2,\ldots, n\} \\ |P| = i}} |S_P| \\ &\stackrel{(ii)}{\leq} o(u_{p,k}^n) + \sum_{i \geq \alpha n}^{n} \binom{n}{i} \left(u_{p,k} + o(1) \right)^i \stackrel{(iii)}{\leq} o(u_{p,k}^n) +  \left(u_{p,k} + 1 + o(1) \right)^n \\ &= \left(u_{p,k} + 1 \right)^{n+o(n)}\,,
\end{align*}
where $(i)$ follows since $|S_P| \leq q^{\alpha n}$ for $|P| \leq \alpha n$.  Since $2 q^{\alpha} = 2 p^{\alpha k} < u_{p,k}$ for every $\alpha < 1/4$ and $k \geq 1$, inequality $(ii)$ holds due to the fact that $|S_P| \leq \left(u_{p,k} + o(1) \right)^{|P|}$.  Finally inequality $(iii)$ is due to the binomial theorem. Therefore the theorem follows.
\end{proof}

We will see that it is possible to improve the bounds given in Theorem \ref{thm:generalBound} using the Slice Rank method for $q = 3^k$ and $k=2,3,4,5$. 

As done before,  let $S = (g_1, g_2, \ldots, g_{\ell})$ be a sequence of elements in $\mathbb{F}_q^n$ with $q = 3^k$ such that for every three elements $g'_1,  g'_2,  g'_3$ of $S$ satisfies $e_{2}(g'_1, g'_2,  g'_3) \neq 0$.  Let $S_1$ be the set obtained from removing the repeated elements in $S$.  By equation \eqref{UpperFor3} we have that $|S_1| \geq |S|/2$.  Now we split $S_1$ in $n+1$ sets $S_1^0,S_1^1,\dots, S_1^{n}$ where $g_i\in S_1^{j}$ if $g_i$ has exactly $j$ coordinates equal to zero.  We note that there exists $j$ such that $$|S_1^{j}|\geq \frac{|S_1|}{n+1} \geq \frac{|S|}{2(n+1)}.$$

Now,  let us recall some definitions and lemmas from \cite{Blog1} and \cite{Blog2}.
\begin{defn}
A function $T:A^k\rightarrow \mathbb{F}$ is said to be a slice if it can be written in the form
$$T(x_1,\dots,x_k)=T_1(x_i)T_2(x_1,\dots,x_{i-1},x_{i+1},\dots,x_k)$$
where $T_1:A\rightarrow \mathbb{F}$ and $T_2: A^{k-1}\rightarrow \mathbb{F}$.
\end{defn}
\begin{defn}
The Slice Rank $srk(T)$ of a general function $T:A^k\rightarrow \mathbb{F}$ is the smallest number $m$ such that $T$ is a linear combination of $m$ slices.
\end{defn}
\begin{lem}[\cite{Blog1}]\label{Lemma1}
Let $A$ be a finite set and $\mathbb{F}$ be a field. Let $T(x,y,z)$ be a function $A\times A\times A\rightarrow \mathbb{F}$ such that $T(x,y,z)\not=0$ if and only if $x=y=z$. Then $srk(T) = |A|$.
\end{lem}

In order to apply Lemma \ref{Lemma1}, we want to consider a function that is zero whenever we consider three different elements of $S_1^j$. In particular, given $x,y,z\in \mathbb{F}_q^n$ we consider
\begin{equation}\label{eq:poleq}
P(x,y,z)=\prod_{i=1}^{n} (1-(x_iy_i+y_iz_i+z_ix_i)^{q-1}).
\end{equation}
\begin{lem}\label{LemmaNostro}
Let us consider the function $P(x,y,z)$ on the restricted domain $S_1^j\times S_1^j\times S_1^j\rightarrow \mathbb{F}_q$ where $q=3^k$. Then $P(x,y,z)\not=0$ if and only if $x=y=z$.
\end{lem}
\proof
Here we have that, if $x,y,z$ are in $S_1^j$, then $P(x,y,z)\not=0$ if and only if $x=y=z$. Indeed, if $x,y$, and $z$ are three different elements of $S_1^j$, they are such that $xy+yz+zx\not=0$ and hence $x_iy_i+y_iz_i+z_ix_i\not=0$ for at least one $i\in [1,n]$. This means that $1-(x_iy_i+y_iz_i+z_ix_i)^{q-1}=0$ and hence $P(x,y,z)=0$.

We note that also if we consider an element $x\in S_1^j$ repeated twice and $z\not=x$, we have that $P(x,x,z)=0$. Indeed, since $x$ and $z$ have the same number of zero components, there exists $i$ such that $x_i\not=0$ and $x_i\not=z_i$. Here we have that
$$x_ix_i+x_iz_i+z_ix_i=x_i^2+2x_iz_i=x_i(x_i-z_i)\not=0$$
since both $x_i-z_i$ and $x_i$ are nonzero.
It follows that $P(x,x,z)=0$. Similarly, we prove that also $P(z,x,x)=0$ and $P(x,z,x)=0$.

Finally, we consider an element $x$ repeated three times. In this case, we have that
$$P(x,x,x)=\prod_{i=1}^{n} (1-(x_ix_i+x_ix_i+x_ix_i)^{q-1})=\prod_{i=1}^{n} (1-(3x_ix_i)^{q-1})=1\not=0.$$
\endproof
As a corollary of Lemmas \ref{Lemma1} and \ref{LemmaNostro} we have that:
\begin{cor}
$$|S|\leq 2(n+1)|S_1^{j}|= 2(n+1)srk(P|_{S_1^{j}\times S_1^{j}\times S_1^{j}}).$$
\end{cor}
Now the goal is to upper bound the $srk(P|_{S_1^{j}\times S_1^{j}\times S_1^{j}})$. The following Lemma will help us to make the first step in this direction.
\begin{lem}[\cite{Blog1}]
Let $A$ be a finite set, $A_1\subseteq A$ and $\mathbb{F}$ be a field. Let $T(x,y,z)$ be a function $A\times A\times A\rightarrow \mathbb{F}$. Then
$$ srk(T|_{A_1 \times A_1 \times A_1})\leq srk(T).$$
\end{lem}
We immediately get the following corollary:
\begin{cor}\label{UpperS}
Considering the function $P$ on the domain $\mathbb{F}_{q}^n\times \mathbb{F}_{q}^n\times \mathbb{F}_{q}^n$, we have that
$$|S|\leq 2(n+1)|S_1^{j}|= 2(n+1)srk(P).$$
\end{cor}
Now we aim to prove that $srk(P)$ improves the bound given in Theorem \ref{thm:generalBound}.  For this purpose, we recall the asymptotic rank theory studied by Tao and Sawin in \cite{Blog2} in the special case of polynomial function (we do not need to consider the very general case of tensor Slice Rank).

Given a polynomial $p(x,y,z)$ whose degree in each variables is at most $\delta$, we define $\Gamma$ as the subset of $\{0,1,\dots,\delta\}^3$ of the triples $(d_1,d_2,d_3)$ such that $x^{d_1}y^{d_2}z^{d_3}$ has a nonzero coefficient in $p$. Hence we state the following proposition derived from \cite{Blog2}.

\begin{prop}\label{Blog2prop}
Let $p(x,y,z)$ be a polynomial and let $\Gamma$ be its support. Then:
$$ srk(\prod_{i=1}^n p(x_i,y_i,z_i))\leq \exp((H(\Gamma)+o(1))n)$$
where
$$H(\Gamma):=\sup_{(X_1,\dots,X_k)} \min(h(X_1),\dots,h(X_k)),$$
$(X_1,\dots,X_k)$ takes values in $\Gamma$ and $h(X)$ is the entropy of the random variable $X$ defined as $-\sum_{\gamma \in \Gamma'} \mathbb{P}[X = \gamma] \log (\mathbb{P}[X = \gamma])$ and $\Gamma'$ is the support of $X$.
\end{prop}
In our case,  we will not find the exact value of $H(\Gamma)$ but we will compute it numerically when $k=2,3,4,5$ solving a convex optimization problem and providing then an upper bound of type $\exp(H(\Gamma))^{(n+o(n))}$ where $\exp(H(\Gamma))$ is strictly smaller than the bounds given in Theorem \ref{thm:generalBound}. For this purpose, we will recall the following theorem from \cite{Bachelor}.
\begin{thm}[Theorem 8 of \cite{Bachelor}]\label{Bachelor}
Let $\Gamma$ be a finite subset of $S \times S \times S$ for some set $S$ and let $\sigma\in Sym(3)$ be a permutation such that for each $a = (a_1,a_2,a_3) \in \Gamma$ also $\sigma(a)=(a_{\sigma(1)},a_{\sigma(2)},a_{\sigma(3)})\in \Gamma$. Then there is a random variable $Y$ taking values in $\Gamma$ such that for all $y\in \Gamma$ we have that $\mathbb{P}[Y=y] = \mathbb{P}[Y=\sigma(y)]$ and
$$ \min (h(Y_1),h(Y_2),h(Y_3))=H(\Gamma).$$
\end{thm}
This theorem essentially ensures us that the value $H(\Gamma)$ is attained as a minimum of the entropy of the marginal variables of some random variable $Y$ and that this variable is invariant under permutations that fix $\Gamma$. We are now ready to state the following theorem.
\begin{thm}\label{main}
Let $q=3^k$, then we have that
$$EGZ(3,\mathbb{F}_{q}^{n},2) \leq \begin{cases} 8.315^{n+o(n)} & \text{for } k=2 \\ 
21.802^{n+o(n)} &\text{for } k=3 \\
58.557^{n+o(n)} &\text{for } k=4 \\
159.812^{n+o(n)} &\text{for } k=5
 \end{cases}.$$
\end{thm}
\proof
We set $p(x,y,z)=(1-(xy+yz+zx)^{q-1})$ and we consider the following polynomial defined in \eqref{eq:poleq}.
$$P(x,y,z)=\prod_{i=1}^n p(x_i,y_i,z_i).$$
Hence we can use Proposition \ref{Blog2prop} to evaluate $srk(P)$.  For $q = 9,27,81,243$,  we compute the support $\Gamma$ of $p$ and then using Theorem \ref{Bachelor} we have been able to compute $H(\Gamma)$ numerically for these cases.
\begin{table}[H]
\centering
\def\arraystretch{1.17}
\begin{tabular}{l|l|l|l|l}
$q$ & $9$ & $27$ & $81$ & $243$ \\
\hline
$H(\Gamma)$ & $2.118$ & $3.082$ & $4.07\,$ & $5.074$ \\
\hline
\end{tabular}
\end{table}

Hence the theorem follows by Corollary \ref{UpperS} and Proposition \ref{Blog2prop}.
\endproof

\begin{rem}
One can prove that $H(\Gamma) < \log q$ for every $q = 3^k \geq 9$, where $\Gamma$ is the support of the polynomial $p(x,y,z)$ defined in Theorem \ref{main}.  The reader can find a proof in a previous version of this paper \cite{CDV2}.
\end{rem}

\begin{rem}
We observe
that, in Theorem \ref{main}, for $q=3$ we obtain a weaker bound than for the other cases of $q$. Indeed, in this case
$$\Gamma=\{(0,0,0),(2,2,0),(0,2,2),(2,0,2),(2,1,1),(1,2,1),(1,1,2)\} $$
and one can easily check that defining $Y$ that has, neatly, the distribution
$$(1/4,1/12,1/12,1/12,1/6,1/6,1/6)$$
over $\Gamma$, $Y_1$, $Y_2$ and $Y_3$ have all uniform distributions.
It follows that, in this case, $H(\Gamma)=\log 3$ and hence our proof fails to provide a better upper bound for $q=3$.
\end{rem}

For the other values of $q$ (i.e. $q > 243$) we have not been able to explicitly evaluate $H(\Gamma)$ since it seems that there are too many variables for this problem to be treated even with the help of a computer. 
\section*{Acknowledgements}
We would like to thank the anonymous reviewer for the simple proof of Proposition \ref{bound} and for pointing out the procedure used in Theorem \ref{thm:generalBound}.
The first author was partially supported by INdAM--GNSAGA.

\end{document}